\documentclass[a4paper,11pt]{article}
\usepackage{multirow}
\usepackage{color}
\usepackage{latexsym}
\usepackage{amssymb}
\usepackage{amsfonts}
\usepackage{amsmath}
\usepackage{amsthm}
\usepackage{indentfirst}
\usepackage{mathrsfs}
\usepackage{verbatim}

\usepackage{tikz}
\usetikzlibrary{patterns}

\usepackage[T1]{fontenc}

\newtheorem {theorem}{Theorem}[section]

\newtheorem {lemma}{Lemma}[section]
\newtheorem {proposition}{Proposition}[section]
\newtheorem {corollary}{Corollary}[section]

\theoremstyle{remark}

\newcommand{\powerset}{\raisebox{.15\baselineskip}{\Large\ensuremath{\wp}}}
\newcommand{\mex}{\textrm{mex}}

\usepackage[latin1]{inputenc}
\usepackage{diagbox}

\usepackage{fullpage}

\author{Antonio Bernini\thanks{Dipartimento di Matematica e Informatica ``U. Dini'', University of Firenze, Firenze, Italy\quad \texttt{\{antonio.bernini, stefano.bilotta, luca.ferrari\}@unifi.it}. Members of INdAM -- GNCS}\and Stefano Bilotta$^\star$\and Giulio Cerbai\thanks{Independent researcher\quad \texttt{giuliocerbai14@gmail.com}}\and Luca Ferrari$^\star$}

\title{The \textrm{mex} statistic on combinatorial structures}

\date{}

\begin{document}
	
	\maketitle
	
	\begin{abstract}
		We extend the notion of $\mex$, which is central in combinatorial number theory, to an arbitrary combinatorial structure, and we prove a general theorem to determine the generating function of the objects having fixed $\mex$. We then study this new $\mex$ statistic for several classical combinatorial structures, by providing the $\mex$ generating function and/or a closed formula for its coefficients in each of the cases.
	\end{abstract}

\section{Introduction}

Given a set $S$ of nonnegative integers, the $\mex$ (\textbf{m}inimal \textbf{ex}cludant) of $S$ is the smallest nonnegative integer not belonging to $S$. In combinatorial game theory, and in particular in the study of impartial normal-play games, a central role is played by the $\mex$ function. Indeed, the classical Sprague-Grundy theorem states that every position in an impartial normal-play game is equivalent (in a properly defined sense) to a position in the Nim game, and the $\mex$ function is essential in proving such a result (we refer the interested reader to \cite{DK} for a gentle yet fully rigorous introduction to game theory in general and combinatorial game theory in particular).

In recent years, the notion of $\mex$ has been extended to integer partitions by Andrews and Newman \cite{AN}. The main focus in \cite{AN}, and in several subsequent papers, is however on number-theoretic properties mainly (for instance, the relationship with the notion of crank). The goal of the present article is to push further this approach, by describing a general notion of $\mex$ for an arbitrary combinatorial structure, and to investigate this new statistic for several classical objects. In order to do this, we first need to formally define what a combinatorial structure is, in a way that is amenable to properly state our definition of the $\mex$ statistic. Our main result is a general methodology to determine the generating function of the objects having fixed $\mex$ of an arbitrary combinatorial structure (with respect to a properly defined notion of \emph{size}), which we prove using an inclusion/exclusion argument. Applying our technique to a given combinatorial structure requires the enumeration of the objects of that structure which, in some sense, do not contain ``pieces'' of a certain type (this statement will be formally explained in the next section). This leads to enumerative issues that, in some cases, have already been explored, so that we can exploit results scattered in the literature to fruitfully apply our methodology.

The notion of $\mex$ can be also interpreted as a way (new, to our knowledge) to generalize the problem of counting objects which do not contain certain ``pieces'' of the smallest possible ``weight'' (which corresponds to having $\mex$ 1). This is a very natural problem that is often addressed for many combinatorial structures, leading to classical number sequences covered in \cite{OEIS}, such as (a variant of) Fibonacci numbers (integer compositions without parts of size 1), Fine numbers (Dyck paths without hills) and Riordan numbers (planar trees without nodes having a single child).
%
%\begin{itemize}
%	\item integer partitions having $\mex$ 1 are those without parts of size 1 (A002865 in \cite{OEIS});
%	\item integer compositions having $\mex$ 1 are those without parts of size 1 (a variant of Fibonacci numbers, A212804 in \cite{OEIS});
%	\item 
%\end{itemize}   

One feature of our technique is that it is particularly efficient from a computational point of view. Indeed, even if, in most of the cases, working ``by hand'' we are able to find simple and readable expressions for generating functions (and their coefficients) only for small values of the $\mex$, with the help of a computer we can instead rather easily derive such information for much larger values. This allows us, among other things, to formulate conjectures which we have been able to solve only in part.

Our paper is organized as follows. In Section \ref{main_def_and_prop} we introduce our model of a combinatorial structure, and we use it to define the notion of $\mex$; moreover, we prove our main theorem for computing the generating function for the objects of a combinatorial structure having fixed $\mex$ (with respect to the size). All the subsequent sections are devoted to the analysis of specific combinatorial structures, namely: integer partitions (Section \ref{IP}), integer compositions (Section \ref{IC}), inversion sequences (Section \ref{IS}), Dyck paths (Section \ref{DP}), set partitions (Section \ref{SP}) and planar trees (Section \ref{PT}). We conclude with a few suggestions for further research work, which are collected in Section \ref{further}.

\section{Combinatorial structures and the mex function}\label{main_def_and_prop}

In this section we introduce a formal definition of combinatorial structure which is particularly convenient to deal with the notion of $\mex$. To be precise, we do not
provide all the details related to the formalism we are going to define, as this would unnecessarily complicate the discussion. Instead, we will limit ourselves to give the essential tools to justify our approach and understand all subsequent results.

We remark that our presentation of a combinatorial structure is closely related to the classical symbolic method as it was introduced by Flajolet and Sedgewick \cite{FS}. However, the main difference of our approach is that the construction of a combinatorial structure is not performed starting from ``atoms'' (or atomic classes, as they are called), but rather starting from some kind of ``molecules'', or ``pieces'' (i.e., sort of aggregates of ``atoms''). This will be relevant when we will define the central notion of $\mex$.  

\bigskip

A \emph{combinatorial structure} $\mathcal{S}$ is completely determined by its \emph{pieces}, its \emph{instances}, and its \emph{properties}. Formally, we denote with $E(\mathcal{S})$ the set of pieces of $\mathcal{S}$. The set of instances of $\mathcal{S}$ is a set $I(\mathcal{S})$ constructed from $E(\mathcal{S})$ using standard set-theoretic operations, such as union, intersection, cartesian product, powerset, etc. Finally, the properties of $\mathcal{S}$ are a list $P(\mathcal{S})$ of conditions that the elements of $I(\mathcal{S})$ are required to obey. Given $X\in I(\mathcal{S})$, we will say that $e \in E(\mathcal S)$ is \emph{contained} in $X$ (or \emph{belongs} to $X$), denoted by $e \vdash X$, when $e$ is among the pieces that are used in the construction of $X$. With this notation the combinatorial structure $\mathcal{S}$ is defined as $$\mathcal{S}=\{X \in I(\mathcal{S}) \, | \, X \text{ satisfies } P(\mathcal{S}) \}.$$
An element of $\mathcal{S}$ is called a \emph{(combinatorial) object}.

\bigskip

Below we illustrate a gallery of examples that will be studied in the following sections. For the objects of each combinatorial structure we will also define a notion of \emph{size}.

\begin{enumerate}
	
	\item \emph{Integer partitions $\mathcal{IP}$.}\quad Let $E(\mathcal{IP})=\mathbf{N}$, the set of natural numbers, and $I(\mathcal{IP})=\bigcup_{k\in \mathbf{N}}\mathbf{N}^k$. An instance of $\mathcal{IP}$ is a $k$-tuple $(\lambda_1 ,\lambda_2 ,\ldots ,\lambda_k )$ of natural numbers (and $\lambda_1 ,\ldots ,\lambda_k$ are the pieces belonging to such an instance). The properties of $\mathcal{IP}$ can be condensed into the single requirement that $\lambda_1 \geq \lambda_2 \geq \cdots \geq \lambda_k >0$, for a given instance $(\lambda_1 ,\lambda_2 ,\ldots ,\lambda_k )$. The size of the integer partition $(\lambda_1 ,\lambda_2 ,\ldots ,\lambda_k )$ is defined as $\lambda_1 +\lambda_2 +\cdots +\lambda_k$.
	
	\item \emph{Integer compositions $\mathcal{IC}$.}\quad The sets $E(\mathcal{IC})$ and $I(\mathcal{IC})$ are defined in the same way as for integer partitions. The only difference relies in the properties, which are expressed by just requiring that $\lambda_1 ,\lambda_2 ,\ldots ,\lambda_k >0$. The notion of size is the same as for integer partitions as well.
	
	\item \emph{Inversion sequences $\mathcal{IS}$.}\quad Again, the sets $E(\mathcal{IS})$ and $I(\mathcal{IS})$ are defined as for integer partitions and compositions. In this case, the required properties are that $\lambda_i \leq i-1$, for all $i$. For inversion sequences, however, we use a different notion of size: the size of the inversion sequence $(\lambda_1 ,\lambda_2 ,\ldots ,\lambda_k )$ is defined to be $k$. 
	
	\item \emph{Dyck paths $\mathcal{DP}$.}\quad Here the basic idea is to identify a Dyck path with the set of its peaks. So the set of pieces is $E(\mathcal{DP})=\mathbf{N}^2$ (just think of them as the coordinates of the peaks). Moreover, $I(\mathcal{DP})=\bigcup_{k\in \mathbf{N}}(\mathbf{N}^2 )^k$, so that an instance is a $k$-tuple $((x_1 ,y_1 ),(x_2 ,y_2 ),\ldots ,(x_k ,y_k ))$ (and each $(x_i ,y_i )$ is a piece contained in such an instance). As for the list of properties of the instances, we leave to the interested reader the easy task of checking that the following requirements completely characterize the sets of possible peaks of a Dyck path:
	\begin{itemize}
		\item $x_i +y_i$ is even, for all $i$;
		\item $x_i <x_{i+1}$, for all $i$;
		\item $y_i >0$, for all $i$;
		\item $x_1 =y_1$, and $|y_{i+1}-y_i |<x_{i+1}-x_i \leq y_{i+1}+y_i$ for all $i$. 
	\end{itemize}

The size of the Dyck path $((x_1 ,y_1 ),(x_2 ,y_2 ),\ldots ,(x_k ,y_k ))$ is given by $\frac{x_k +y_k}{2}$ (the usual semilength of the path). 

	\item \emph{Set partitions $\mathcal{SP}$.}\quad Let $E(\mathcal{SP})=\powerset_{fin}(\mathbf{N})$ (the set of all finite subsets of $\mathbf{N}$), and $I(\mathcal{SP})=\bigcup_{k\in \mathbf{N}}\powerset_{fin}(\mathbf{N})^k$. Thus an instance of $\mathcal{SP}$ is a $k$-tuple $(B_1 ,B_2 ,\ldots ,B_k )$ of finite subsets of $\mathbf{N}$. The properties for $\mathcal{SP}$ are the following:
	\begin{itemize}
		\item $\min B_i <\min B_{i+1}$, for all $i$;
		\item $B_i \neq \emptyset$, for all $i$;
		\item there exists $n\in \mathbf{N}$ such that $\bigcup_{i=1}^{k}B_i =\{ 1,2,\ldots ,n\}$;
		\item $B_i \cap B_j =\emptyset$, for all $i,j$ with $i\neq j$.
	\end{itemize}

The size of the set partition $(B_1 ,B_2 ,\ldots ,B_k )$ is given by $|B_1 |+|B_2 |+\cdots +|B_k |=\left| \bigcup_{i=1}^{k}B_i\right|$. 

	\item \emph{Planar trees $\mathcal{PT}$.}\quad Every planar tree is completely determined by the ordered lists of the children of each vertex. So the pieces here are essentially tuples where the first entry is a vertex and the successive entries are its children (in the required order). In order to correctly formalize this basic idea, we assume that the vertices of a tree are labelled according to the well known breadth first search procedure (so that the set of labels is an initial segment of $\mathbf{N})$. We set $E(\mathcal{PT})=\powerset_{fin}(\mathbf{N})$ and $I(\mathcal{PT})=\bigcup_{k\in \mathbf{N}}\powerset_{fin}(\mathbf{N})^k$ (i.e., pieces and instances are the same as for set partitions). Given $(A_1 ,A_2 ,\ldots ,A_k )\in I(\mathcal{PT})$, the properties we require for it in order to be an object of $\mathcal{PT}$ (i.e., a planar tree) are the following:
	\begin{itemize}
		\item $\min A_i =i$, for all $i$;
		\item $A_i \setminus \{ i\}$ is an interval of natural numbers, for all $i$;
		\item $(A_1 \setminus \{ 1\},A_2 \setminus \{ 2\},\ldots ,A_k \setminus \{ k\} )$ is a set partition of $\{ 2,3,\ldots k\}$.
	\end{itemize}

The size of the planar tree $(A_1 ,A_2 ,\ldots ,A_k )$ is simply $k$ (the number of vertices of the tree).

\end{enumerate} 

\bigskip

Given a combinatorial structure $\mathcal{S}$, any map $w:E(\mathcal{S})\rightarrow \mathbf{N}$ is called a \emph{weight}, and the pair $(\mathcal{S},w)$ is called a \emph{weighted combinatorial structure}. For a weighted combinatorial structure $(\mathcal{S},w)$, we introduce the $\mex$ statistic on its objects as follows. Given $X\in \mathcal{S}$, we define $\mex (X)$ as the minimum $k\in \mathbf{N}$ such that $X$ does not contain any piece of weight $k$. Formally, 
\[
\mex(X)=\min_{e \vdash X}\{ k\in \mathbf{N}\, |\, w(e)\neq k\}.
\]

Notice that, if we take the set of natural numbers $\mathbf{N}$ as pieces and $\powerset (\mathbf{N})\setminus \{ \mathbf{N}\}$ as instances, imposing no properties yields the combinatorial structures of all proper subsets of $\mathbf{N}$. Setting the weight of a piece equal to itself (i.e., $w(n)=n$, for all $n\in \mathbf{N}$), the resulting notion of $\mex$ is precisely the classical one that can be found in combinatorial game theory.

\bigskip

Our goal is to propose a general methodology for computing the generating function for the objects having fixed $\mex$ of a weighted combinatorial structure (with respect to the size). Our approach is based on a suitable application of the classical principle of inclusion-exclusion. From now on, given a combinatorial structure $\mathcal{S}$, the set of objects of $\mathcal{S}$ having size $n$ will be denoted $\mathcal{S}_n$. 

\begin{theorem}\label{main_th}
	Let $(\mathcal{S},w)$ be a weighted combinatorial structure, in which a notion of size for its objects is also defined. Given $T\subseteq [m]=\{ 1,2,\ldots ,m\}$, denote with $G_T (x)$ the generating function for the objects of $\mathcal{S}$ which do not contain any piece having weight $i$, for all $i\in T$ (with respect to the size). Then the generating function $\Gamma_m (x)$ for the objects of $\mathcal{S}$ having \emph{mex} $m$ (with respect to the size) is given by
	\begin{equation}\label{general_mex}
	\Gamma_m (x)=\sum_{k=1}^{m}(-1)^{k+1}\sum_{{T\subseteq [m]\atop |T|=k}\atop m\in T}G_T (x).
	\end{equation}  
\end{theorem}

\begin{proof}
	Denote with $F_m (x)$ the generating function for the objects of $\mathcal{S}$ containing at least one piece of weight $i$, for all $i\leq m$ (with respect to the size). If we define $A_i ^{(n)} =\{ X\in \mathcal{S}_n\, |\, w(e)\neq i, \textnormal{ for all }e\vdash X\}$, then, by the principle of inclusion-exclusion, we have that
	\begin{align*}
	F_m (x)&=\sum_{n\geq 0}\left| A_1 ^{(n)'}\cap A_2 ^{(n)'}\cap \cdots \cap A_m ^{(n)'}\right| x^n \\
		   &=\sum_{n\geq 0}\left( \sum_{k=0}^{m}(-1)^k\sum_{T\subseteq [m]\atop |T|=k}\left| \bigcap_{i\in T}A_i ^{(n)}\right|    \right) x^n \\
		   &=\sum_{k=0}^{m}(-1)^k \sum_{T\subseteq [m]\atop |T|=k}\left( \sum_{n\geq 0} \left| \bigcap_{i\in T}A_i ^{(n)}\right| x^n \right)\\
		   &=\sum_{k=0}^{m}(-1)^k \sum_{T\subseteq [m]\atop |T|=k}G_T (x) .
	\end{align*}
	
	Observe that $\mex(X)=m$ if and only if, for all $i=1,\ldots,m-1$, there exists $e \in E(\mathcal S)$, $e \vdash X$, such that $w(e)=i$ and there is no $e \in E(\mathcal S)$, $e \vdash X$, such that $w(e)=m$. In other words, the set of objects of $\mathcal{S}$ having $\mex$ $m$ can be obtained by taking all the objects of $\mathcal S$ with at least one piece of weight $i$, for all $i=1,\ldots,m-1$, and subtracting those also having at least one piece of weight $m$. 
	Therefore, in terms of generating functions, we have $$\Gamma_m (x)=F_{m-1}(x)-F_m (x).$$ 
	
	From here, we finally get
	\begin{align*}
	\Gamma_m (x)&=\sum_{k=0}^{m-1}(-1)^k \sum_{T\subseteq [m-1]\atop |T|=k}G_T (x) -\sum_{k=0}^{m}(-1)^k \sum_{T\subseteq [m]\atop |T|=k}G_T (x)\\
	&=\sum_{k=1}^{m}(-1)^{k+1}\sum_{{T\subseteq [m]\atop |T|=k}\atop m\in T}G_T (x),
	\end{align*}
	as desired.
\end{proof}

In the rest of the paper we will exploit Theorem \ref{main_th} and formula (\ref{general_mex}) to find information on the $\mex$ statistic for each of the combinatorial structures we have introduced at the beginning of the present section. Depending on the cases, we will describe the generating function or a closed form for the number of objects having $\mex$ $m$ with respect to the size. In general, working ``by hand'', we will be able to find explicit expressions only for small values of $m$, however our approach is particularly suited for automated computations, and using a computer we can find formulas for many values of $m$ in most of the cases.

As a general notation, we will write $\Gamma_{m}^{\mathcal{S}}(x)=\sum_{n\geq 0}\gamma_{n,m}^{\mathcal{S}}x^n$ for the generating function of the objects having $\mex$ $m$ of the (weighted) combinatorial structure $\mathcal{S}$ with respect to the size; in particular, $\gamma_{n,m}^{\mathcal{S}}$ is the number of objects of $\mathcal{S}$ of size $n$ having $\mex$ $m$.  

\section{Integer partitions}\label{IP}

On the combinatorial structure $\mathcal{IP}$ of integer partitions we introduce the weight $w$ by setting $w(n)=n$, for all $n\in \mathbf{N}$. According to our definition, the associated $\mex$ statistic is then defined by setting $\mex (\lambda  )$ to be the smallest positive integer that is not a part of $\lambda$. 

The notion of minimal excludant for integer partitions is not new, and has been introduced by Andrews and Newman in \cite{AN}. The main focus in \cite{AN}, and in several subsequent papers, such as \cite{BaSi,dSS,HSS,KBEM}, is however on number-theoretic properties mainly. Prior to Andrews and Newman's paper, Grabner and Knopfmacher \cite{GK} investigated the $\mex$ statistic on integer partitions under the name ``smallest gap in partitions'', and also determined some genuinely combinatorial properties of it. In particular, it turns out that the generating function of integer partitions having $\mex$ $m$ is rather easy to compute. The proof below is a direct consequence of a result in \cite{GK}, and clearly shows the similarity with our approach.

\begin{proposition}
	The generating function $\Gamma_{m}^{\mathcal{IP}}(x)$ of integer partitions having $\mex$ $m$ is given by
	\begin{equation}\label{gf_mex_int_part}
	\Gamma_{m}^{\mathcal{IP}}(x)=\left( x^{\binom{m}{2}}-x^{\binom{m+1}{2}}\right) P(x),
	\end{equation}
	where $P(x)=\prod_{i\geq 1}\frac{1}{1-x^i}$ is the well known generating function of integer partitions.
\end{proposition}

\begin{proof}
	Denote with $F_{m}^{\mathcal{IP}}(x)$ the generating function of integer partitions having $\mex$ at least $m+1$. Recall that, in the generating function $P(x)$ of integer partitions, the factor $\frac{1}{1-x^i}$ keeps track of the parts equal to $i$. Thus, since $F_{m}^{\mathcal{IP}}(x)$ counts integer partitions having at least one part of size $i$, for all $i\leq m$, we get immediately
	\[
	F_{m}^{\mathcal{IP}}(x)=\prod_{i=1}^{m}\frac{x^i}{1-x^i}\prod_{i\geq m+1}\frac{1}{1-x^i}=x^{\binom{m+1}{2}}P(x),
	\]
	hence
	\[
	\Gamma_{m}^{\mathcal{IP}}(x)=F_{m-1}^{\mathcal{IP}}(x)-F_{m}^{\mathcal{IP}}(x)=\left( x^{\binom{m}{2}}-x^{\binom{m+1}{2}}\right) P(x),
	\]
	as desired.
\end{proof}

Notice that the expression given above for $F_{m}^{\mathcal{IP}}(x)$ (which in this case is very easy to compute in a direct way) can also be obtained in the context of our framework. Indeed, the generating function $G_{T}^{\mathcal{IP}}(x)$ of integer partitions having no parts of size $i$, for all $i\in T\subseteq [m]$, can be computed by removing from $P(x)$ all factors relating to the forbidden parts, hence $G_{T}^{\mathcal{IP}}(x)=P(x)\prod_{i\in T}(1-x^i)$. As a consequence, we get
\begin{align*}
F_{m}^{\mathcal{IP}}(x)&=\sum_{k=0}^{m}(-1)^k \sum_{T\subseteq [m]\atop |T|=k}G_{T}^{\mathcal{IP}}(x)\\
						&=P(x)\cdot \sum_{k=0}^{m}(-1)^k \sum_{T\subseteq [m]\atop |T|=k}\prod_{i\in T}(1-x^i)\\
						&=P(x)\cdot \prod_{i=1}^{m}(1-(1-x^i ))=x^{\binom{m+1}{2}}P(x).
\end{align*}

The sequence counting integer partitions of size $n$ having $\mex$ $m$ is recorded as sequence A264401 in \cite{OEIS}, where the first values of the corresponding table are also shown.

\section{Integer compositions}\label{IC}

Using the same weight as for integer partitions, the $\mex$ of an integer composition $\lambda$ is again the smallest positive integer that is not a part of $\lambda$. 

To the best of our knowledge, the notion of $\mex$ for integer compositions (from now on, simply compositions) appears to be new. Although in general the combinatorics of integer partitions is much more difficult than the combinatorics of compositions, the determination of the generating functions for compositions having fixed $\mex$ requires slightly more care. In fact, we are not able to give a compact form like in (\ref{gf_mex_int_part}). Nevertheless we can find some interesting information on the nature of such generating functions (as well as some explicit formulas for small values of the $\mex$), showing in particular that, even if a general closed formula is difficult to describe, the generating functions for compositions are simpler than those for integer partitions. The proof of the next proposition reveals the similarities between our methodology (in the specific case of integer compositions) and a classical argument that can be found, for instance, in \cite{BR}.

\begin{proposition}
	The generating function $\Gamma_{m}^{\mathcal{IC}}(x)$ of compositions having $\mex$ $m$ is rational, for every $m\geq 1$.  
\end{proposition}

\begin{proof}
	Let $m\geq 1$. We begin by determining a functional equation satisfied by the generating function $G_{T}^{\mathcal{IC}}(x)$ of compositions having no parts of size $i$, for all $i\in T\subseteq [m]$. Each of such composition is either the empty composition or it can be factored by concatenating its first part with the composition determined by the subsequent parts. Denoting with $H_T (x)$ the generating function of the sequence of positive natural numbers not belonging to $T$, we thus get the equation
	\[
	G_{T}^{\mathcal{IC}}(x)=1+H_T (x)G_{T}^{\mathcal{IC}}(x).
	\]
	
	Since clearly $H_T (x)=\frac{x}{1-x}-\sum_{i\in T}x^i$, we thus get
	\[
	G_{T}^{\mathcal{IC}}(x)=\frac{1}{1-H_T (x)}=\frac{1-x}{1-2x+\sum_{i\in T}(x^i (1-x))}.
	\]
	
	Applying formula (\ref{general_mex}) from Theorem \ref{main_th} allows us to conclude.	
\end{proof}

The expression for $G_{T}^{\mathcal{IC}}(x)$ found in the previous proposition allows us to compute the generating functions $\Gamma_{m}^{\mathcal{IC}}(x)$ for several values of $m$. The next corollary records a few small cases that can be handled without the use of a computer.

\begin{corollary}
	\begin{enumerate}
		\item $\Gamma_{1}^{\mathcal{IC}}(x)=\frac{1-x}{1-x-x^2}$, that is the generating function of (shifted) Fibonacci numbers (sequence A212804 in \cite{OEIS}).
		\item $\Gamma_{2}^{\mathcal{IC}}(x)=\frac{x(1-x)^2}{(1-2x+x^2 -x^3 )(1-x-x^3 )}$, whose sequence of coefficients starts $0,1,1,1,3,6,10,18,33$, $59,105,187,\ldots$ and is not recorded in \cite{OEIS}.
		\item $\Gamma_{3}^{\mathcal{IC}}(x)=\frac{x^3 (1-x)^3  (2-3x+x^3-2x^4)}{(1-x+x^2)(1-x-x^2)(1 -2x +x^3-x^4)(1-x-x^2+x^3-x^4) (1-x-x^4)}$, whose sequence of coefficients starts $0,0,0,2,3,7,11,26,50, 104,197,387,\ldots$ and is not recorded in \cite{OEIS}.
	\end{enumerate}
\end{corollary}

\begin{proof}
	\begin{enumerate}
		\item Applying (\ref{general_mex}) with $m=1$ and $G_{T}^{\mathcal{IC}}(x)$ as in the above proposition, we get
		\[
		\Gamma_{1}^{\mathcal{IC}}(x)=G_{\{ 1\}}^{\mathcal{IC}}(x)=\frac{1-x}{1-2x+x(1-x)}=\frac{1-x}{1-x-x^2},
		\]
		as desired.
		\item Applying (\ref{general_mex}) with $m=2$ and $G_{T}^{\mathcal{IC}}(x)$ as in the above proposition, we get
		\begin{align*}
		\Gamma_{2}^{\mathcal{IC}}(x)&=G_{\{ 2\}}^{\mathcal{IC}}(x)-G_{\{ 1,2\}}^{\mathcal{IC}}(x)=\frac{1-x}{1-2x+x^2 (1-x)}-\frac{1-x}{1-2x+x(1-x) +x^2 (1-x)} \\
		&=\frac{x(1-x)^2}{(1-2x+x^2 -x^3 )(1-x-x^3 )},
		\end{align*}
	as desired.
	\item Applying (\ref{general_mex}) with $m=3$ and $G_{T}^{\mathcal{IC}}(x)$ as in the above proposition, we get
	\begin{align*}
		\Gamma_{3}^{\mathcal{IC}}(x)&=G_{\{ 3\}}^{\mathcal{IC}}(x)-G_{\{ 1,3\}}^{\mathcal{IC}}(x)-G_{\{ 2,3\}}^{\mathcal{IC}}(x)+G_{\{ 1,2,3\}}^{\mathcal{IC}}(x)=\frac{1-x}{1 - 2x + x^3 - x^4}-\\ 
		&-\frac{1-x}{1 - x - x^2 + x^3 - x^4} 
		- \frac{1-x}{1 - 2x + x^2 - x^4} 
		+ \frac{1-x}{1 - x - x^4}= \\
		&=\frac{x^3 (1-x)^3  (2-3x+x^3-2x^4)}{(1-x+x^2)(1-x-x^2)(1 -2x +x^3-x^4)(1-x-x^2+x^3-x^4) (1-x-x^4)},
	\end{align*}
	as desired.
	\end{enumerate}
\end{proof}

Using a computer, it is not difficult to determine $\Gamma_{m}^{\mathcal{IC}}(x)$ for larger values of $m$, but the resulting expressions are not particularly nice-looking. We can however explicitly write down several coefficients $\gamma_{n,m}^{\mathcal{IC}}$ of the generating functions $\Gamma_{m}^{\mathcal{IC}}(x)$ (counting compositions of $n$ having mex $m$), as reported in Table \ref{composition}.

\begin{table}[h!]
	\begin{center}
		\begin{tabular}{c|cccccccc}
			\backslashbox{$n$}{$m$} & 0 & 1 & 2 & 3 & 4 & 5 & 6 & 7 \\
			\hline
			0  & 1  &      &       &        &        &        &       &    \\
			1  & 0  & 0    & 1	   &	    & 	     &        &       &    \\
			2  & 0  & 1    & 1     &        &        &        &       &    \\
			3  & 0  & 1    & 1     & 2      &        &        &       &    \\
			4  & 0  & 2    & 3     & 3      &        &        &       &    \\
			5  & 0  & 3    & 6     & 7      &        &        &       &    \\
			6  & 0  & 5    & 10    & 11     & 6      &        &       &    \\
			7  & 0  & 8    & 18    & 26     & 12     &        &       &    \\
			8  & 0  & 13   & 33    & 50     & 32     &        &       &    \\
			9  & 0  & 21   & 59    & 104    & 72     &        &       &    \\
			10 & 0  & 34   & 105   & 197    & 152    & 24     &       &    \\
			11 & 0  & 55   & 187   & 387    & 335    & 60     &       &    \\
			12 & 0  & 89   & 332   & 738    & 709    & 180    &       &    \\
			13 & 0  & 144  & 588   & 1425   & 1489   & 450    &       &    \\
			14 & 0  & 233  & 1040  & 2711   & 3092   & 1116   &       &    \\
			15 & 0  & 377  & 1837  & 5170   & 6336   & 2544   & 120   &    \\
			16 & 0  & 610  & 3241  & 9791   & 12894  & 5872   & 360   &    \\
			17 & 0  & 987  & 5713  & 18543  & 26061  & 13032  & 1200  &    \\
			18 & 0  & 1597 & 10063 & 34994  & 52380  & 28738  & 3300  &    \\
			19 & 0  & 2584 & 17714 & 65990  & 104804 & 62208  & 8844  &    \\
			20 & 0  & 4181 & 31166 & 124186 & 208843 & 133712 & 22200 &    \\
			21 & 0  & 6765 & 54810 & 233505 & 414718 & 283938 & 54120 & 720\\
		\end{tabular}
	\end{center}
	\caption{The number $\gamma_{n,m}^{\mathcal{IC}}$ of compositions of $n$ having $\mex$ $m$, for small values of $n$ and $m$.} \label{composition}
\end{table}

To conclude this section, we would like to point out that our method is particularly useful for a computer-based approach to explicitly compute the generating functions $\Gamma_{m}^{\mathcal{IC}}(x)$ for several values of $m$.
On the other hand, also looking at the coefficients of Table \ref{composition}, we are able to provide a nice closed form for such coefficients, which is not particularly effective from a computational point of view, but it is rather elegant and combinatorially meaningful.

\begin{proposition}
	Setting $\Gamma_{m}^{\mathcal{IC}}(x)=\sum_{n \geq 0} \gamma_{n,m}^{\mathcal{IC}} x^n$, for $n < \binom{m}{2}$
	 we have $\gamma_{n,m}^{\mathcal{IC}}=0$, whereas for $n=\binom{m}{2}+k$, with $k \geq 0$, we get
	
	\begin{equation}\label{IC_coeff}
		\gamma_{n,m}^{\mathcal{IC}}=\sum_{\alpha_1, \ldots, \alpha_k \in \mathbf{N}, \ \alpha_m=0 \atop \alpha_1+2\alpha_2+\cdots+k\alpha_k=k} \frac{(m-1+\alpha_1+\cdots+\alpha_k)!}{\displaystyle\prod_{i=1}^{m-1}(\alpha_i+1)!\prod_{i=m+1}^{k}\alpha_i!}. 
	\end{equation}
\end{proposition}    

\begin{proof}
	Let $\tau$ be a composition of $n$ having $\mex$ $m$. By definition, the multiset of the parts of $\tau$ must contain at least one occurrence of each of the integers $1,2,\ldots,m-1$. Therefore, when $n < \binom{m}{2}$, we have $\gamma_{n,m}^{\mathcal{IC}}=0$.   
	
	Now assume that $n=\binom{m}{2}+k$, with $k \geq 0$. In this case the multiset of parts of $\tau$ consists of one occurrence of each of the integers $1,2,\ldots, m-1$ together with other parts chosen among the integers $1,2,\ldots,k$, of course excluding $m$ and possibly including further occurrences of $1,2,\ldots,m-1$, summing to $k$. Suppose that the multiset of the added parts contains $\alpha_i \geq 0$ occurrences of $i$, for all $i \leq k$, so that $\sum_{i=1}^{k} i \alpha_i=k$. In particular, note that $\alpha_m=0$. As it is true for every composition, $\tau$ is uniquely determined by a permutation of its parts, which are $m-1+\alpha_1+\cdots+\alpha_k$, up to a permutation of each subset of equal parts. Since the number of parts of $\tau$ equal to $i$ is $\alpha_i +1$ when $i < m$ and is $\alpha_i$ when $i > m$, we get precisely formula (\ref{IC_coeff}).  
\end{proof}

The condition $\alpha_m=0$ in formula (\ref{IC_coeff}) only applies when $m \leq k$. Moreover, the special case $k=0$ does not precisely fit into formula (\ref{IC_coeff}) because in such a case the summation is empty. However, we can recover it by taking a single summand in which the $\alpha_i$'s are set equal to $0$, which gives $(m-1)!$.

\section{Inversion sequences}\label{IS}

An \emph{inversion sequence of size $n$} is a $n$-tuple $(x_1 ,x_2 ,\ldots ,x_n)$ of nonnegative integers satisfying the conditions $x_i \leq i-1$, for all $i\leq n$. Inversion sequences are a useful encoding of permutations (see for instance \cite{St}), whose combinatorial properties have extensively been investigated, especially in the context of patterns (see \cite{AC,CGGHL,T} to cite only a few of many papers dealing with inversion sequences). Referring to the presentation of inversion sequences given in Section \ref{main_def_and_prop}, and using the same weight as for integer partitions and compositions, the $\mex$ of the inversion sequence $(x_1 ,x_2 ,\ldots ,x_n)$ is defined as the minimum integer $m$ such that $x_i \neq m$ for all $i\leq n$. Our aim is to find information on the number of inversion sequences of size $n$ having $\mex$ $m$. We will exploit our general approach, however in this specific case it is much more convenient to describe the coefficients of the generating function $\Gamma_m ^{\mathcal{IS}}(x)$ rather than the generating function itself. This is essentially due to the fact that the expression for the generating functions $G_T ^{\mathcal{IS}}(x)$ is not particularly meaningful, whereas the coefficient of $x^n$ in $\sum_{{T\subseteq [m]\atop |T|=k}\atop m\in T}G_T ^{\mathcal{IS}}(x)$ turns out to have a very interesting and compact description (at least for all $n\geq m$).
In what follows we will need the \emph{Stirling numbers of the second kind} $S_{n,k}$, which count the number of set partitions of a set having $n$ elements into $k$ blocks. The next lemma records a well known fact concerning Stirling numbers.

\begin{lemma}\label{column_stirling}
	Let $g(n,k)$ denote the sum of all possible products of $n-k$ (not necessarily distinct) integers in $[k]$. Then $g(n,k)=S_{n,k}$.
\end{lemma}

\begin{proof} (sketch)
	Fix $k\in \mathbf{N}$. The generating function $\sum_{n\geq k}g(n,k)x^n$ can be expressed as follows:
	\begin{align*}
		\sum_{n\geq k}g(n,k)x^n &=\sum_{n\geq 0}\left( \sum_{j_1 ,\ldots ,j_k \geq 0\atop j_1 +\cdots +j_k =n}1^{j_1}2^{j_2}\cdot \ldots \cdot k^{j_k} \right) x^{n+k} \\
		&=x^k \cdot \displaystyle\prod_{j=1}^{k}\frac{1}{1-jx}.
	\end{align*}

In \cite[Chapter 1, (1.94c)]{St} it is stated and then proved (using the standard recursion for the Stirling numbers of the second kind) that $\sum_{n\geq k}S_{n,k}x^n =x^k \cdot \displaystyle\prod_{j=1}^{k}\frac{1}{1-jx}$, which concludes the proof.
\end{proof}   

\begin{proposition}\label{coeff_G}
	Let $m,k\in \mathbf{N}$, with $k\leq m$, and assume that $n\geq m$. The coefficient of $x^n$ in $\sum_{{T\subseteq [m]\atop |T|=k}\atop m\in T}G_T^{\mathcal{IS}} (x)$ is 
	\[
	(m-k+1)S_{m,m-k+1}(n-k)!.
	\]
\end{proposition}

\begin{proof}
	Suppose first that $n>m$. Fix $k$ integers $i_1 ,i_2 ,\ldots ,i_k$ such that $1\leq i_1 <i_2 <\cdots <i_k \leq n-1$, with $i_k =m$. An inversion sequence $(x_1 ,x_2 ,\ldots ,x_n )$ of size $n$ which does not contain any term in the set $\{ i_1 ,i_2 ,\ldots ,i_k \}$ can be uniquely determined by choosing the term $x_i$ in $\alpha_i$ possible ways, according to the rules in the following table (for all $i\leq n$):
	
	\bigskip
	
		\begin{tabular}{c|c}
			$x_i$ & $\alpha_i$ \\
			\hline
			$x_1$ & 1 \\
			$x_2$ & 2 \\
			$\cdots$ & $\cdots$ \\
			$x_{i_1}$ & $i_1$ \\
			$\mathbf{x_{{i_1}+1}}$ & $\mathbf{i_1}$ \\
			$x_{{i_1}+2}$ & $i_1 +1$ \\
			$\cdots$ & $\cdots$ \\
			$x_{i_2}$ & $i_2 -1$ \\
			$\mathbf{x_{{i_2}+1}}$ & $\mathbf{i_2 -1}$ \\
			$x_{{i_2}+2}$ & $i_2$ \\
			$\cdots$ & $\cdots$ \\
			$x_{i_k}$ & $i_k -(k-1)$ \\
			$\mathbf{x_{{i_k}+1}}$ & $\mathbf{i_k -(k-1)}$ \\
			$x_{{i_k}+2}$ & $i_k -k$ \\
			$\cdots$ & $\cdots$ \\
			$x_n$ & $n-k$ \\
		\end{tabular}
			
	\bigskip
				
%	\begin{table}[h]
%		\centering
%		\footnotesize %oppure \scriptsize
%	\begin{tabular}{cccccccccccccccc}
%		$x_1$ & $x_2$ & $\cdots$ & $x_{i_1}$ & $x_{{i_1}+1}$ & $x_{{i_1}+2}$ & $\cdots$ & $x_{i_2}$ & $x_{{i_2}+1}$ & $x_{{i_2}+2}$ & $\cdots$ & $x_{i_k}$ & $x_{{i_k}+1}$ & $x_{{i_k}+2}$ & $\cdots$ & $x_n$ \\
%		\hline
%		1 & 2 & $\cdots$ & $i_1$ & $i_1$ & $i_1 +1$ & $\cdots$ & $i_2 -1$ & $i_2 -1$ & $i_2$ & $\cdots$ & $i_k -(k-1)$ & $i_k -(k-1)$ & $i_k -k$ & $\cdots$ & $n-k$ \\
%	\end{tabular}
%	\end{table}
	
Indeed, in general the element $x_i$ can be chosen in $i$ ways (it can be any of the integers between 0 and $i-1$); however, since the integers $i_1 ,\ldots ,i_k$ have to be avoided, the number of choices for each of the elements in boldface remains the same as the number of choices for the preceding element. Therefore the number of inversion sequences of size $n$ without $i_1 ,\ldots ,i_k$ is $(n-k)!i_1 (i_{2}-1)\cdots (i_k -k+1)$. Notice that the collection of integers $i_1 ,i_2 -1 ,\ldots ,i_k -k+1=m-k+1$ is a multiset of $k$ (not necessarily distinct) integers in $[n-k]$ whose maximum is $m-k+1$.

The coefficient of $x^n$ in $\sum_{{T\subseteq [m]\atop |T|=k}\atop m\in T}G_T (x)$ is obtained by taking the sum of the number of inversion sequences avoiding a set $T$ of elements of cardinality $k$ containing $m$, for $T$ running over all sets of cardinality $k$ of $[m]$. Thanks to the above considerations, this is given by $(n-k)!(m-k+1)g(m,m-k+1)$, where $g(n,k)$ is defined as in Lemma \ref{column_stirling}. Using the same lemma we thus get the thesis.

Regarding the case $n=m$, we can repeat exactly the same argument, observing that the number of inversion sequences of size $m$ without $i_i ,\ldots ,i_{k-1} $ is $(m-k+1)!i_1 (i_2 -1)\cdots (i_{k-1}-k+2)$. The collection of integers $i_1 ,i_2 -1 ,\ldots ,i_{k-1}-k+2$ is a multiset of $k-1$ (not necessarily distinct) integers in $[n-k+1]$. Therefore the coefficient of $x^m$ in $\sum_{{T\subseteq [m]\atop |T|=k}\atop m\in T}G_T (x)$ is given by $(m-k+1)!S_{m,m-k+1}$, which agrees with our formula when $n=m$.
\end{proof}

From the above result we immediately deduce a closed form for the number of inversion sequences of size $n$ having $\mex$ $m$.

\begin{proposition}
	The coefficient of $x^n$ in $\Gamma_m ^{\mathcal{IS}}(x)$ is 0 when $n<m$, 1 when $n=m$, and
	\[
	\sum_{k=1}^{m}(-1)^{k+1}(m-k+1)S_{m,m-k+1}(n-k)!
	\]
	when $n>m$.
\end{proposition}

\begin{proof}
	When $n<m$ there are no inversion sequences of size $n$ having $\mex$ $m$. When $n\geq m$, using Theorem \ref{main_th} and Proposition \ref{coeff_G}, we get that the coefficient of $x^n$ in $\Gamma_m ^{\mathcal{IS}}(x)$ is precisely as desired.
\end{proof}

 Notice that, in the particular case $n=m$, there is exactly one inversion sequence of size $m$ and having $\mex$ $m$ (which is $(0,1,2,\ldots ,m-1)$), and our formula gives the expression $\sum_{k=1}^{m}(-1)^{m-k}k!S_{m,k}$. Indeed, it is well known that such a sum is equal to 1: just set $x=-1$ in the classical polynomial identity $x^m =\sum_{k=0}^{m}S_{m,k}(x)_k$.
 
 \bigskip
 
 The number of inversion sequences of size $n\geq m$ having $\mex$ $m$ is reported in the table below, for small values of $m$.
 
 \bigskip
 
 \begin{tabular}{c|l}
 	$\mex$ & number of inversion sequences of size $n\geq m$ \\
 	\hline
 	$m=1$ & $\gamma_{n,1}^{\mathcal{IS}}=(n-1)!$ \\
 	$m=2$ & $\gamma_{n,2}^{\mathcal{IS}}=2(n-1)!-(n-2)!$ \\
 	$m=3$ & $\gamma_{n,3}^{\mathcal{IS}}=3(n-1)!-6(n-2)!+(n-3)!$ \\
 	$m=4$ & $\gamma_{n,4}^{\mathcal{IS}}=4(n-1)!-18(n-2)!+14(n-3)!-(n-4)!$ \\
 	$m=5$ & $\gamma_{n,5}^{\mathcal{IS}}=5(n-1)!-40(n-2)!+75(n-3)!-30(n-4)!+(n-5)!$ \\
 \end{tabular}

\bigskip

%Looking at the first values of $\gamma_{n,m}^{\mathcal{IS}}$, it turns out that the resulting table appears to match the table recorded as A056151 in \cite{OEIS} (with reversed rows). We have not been able to prove this striking coincidence.
%
%\bigskip

We also remark that $\gamma_{n,1}^{\mathcal{IS}}$ is, by definition, the number of inversion sequences having no 1's. As we mentioned at the beginning of this section, we can encode permutations using inversion sequences. One way to do it is as follows: given a permutation $\pi =\pi_1 \pi_2 \cdots \pi_n$ of length $n$, its inversion sequence $(x_1 ,x_2 ,\ldots ,x_n)$ is defined by setting $x_i =|\{ j\, |\, j<i, \pi_j >\pi_i \} |$ (i.e., $x_i$ counts the number of elements of $\pi$ to the left of $\pi_i$ and forming an inversion with $\pi_i$). With this definition, it is easy to realize that an inversion sequence having $\mex$ 1 corresponds to a permutation avoiding the mesh pattern
$\vcenter{\hbox{\begin{tikzpicture}[scale=.3]
		\fill[black!20] (0,1) rectangle (2,3);
		\foreach \x in {1,2} {
			\fill (\x,3-\x) circle (6pt);
			\draw (\x,0) -- (\x,3); 
			\draw (0,\x) -- (3,\x);}
		\end{tikzpicture}}}$.
As reported in \cite{HJSVU}, such permutations are indeed counted by $(n-1)!$ (our mesh pattern corresponds to the complement of the mesh pattern labeled 8 in Table 3, page 12 of the mentioned paper, to which we refer the reader for the notion of mesh pattern). Using a similar argument, it is not too difficult to express the property of having $\mex$ $m$ for an inversion sequence in terms of containment and avoidance of certain mesh patterns for the associated permutation, but we have not introduced the tools to give an appropriate description of such patterns here (however the interested readers can easily find the required patterns by themselves).  

\section{Dyck paths}\label{DP}

Using our presentation of Dyck paths, and defining the weight of a piece as $w(x,y)=y$, the $\mex$ of a Dyck path $P$ is the smallest positive integer $m$ such that $P$ does not have a peak having height $m$. In order to use Theorem \ref{main_th} to compute the generating functions $\Gamma_{m}^{\mathcal{DP}}(x)$, we need to count Dyck paths whose set of peaks avoids a given set $T$ of forbidden heights (i.e., to know the generating functions $G_{T}^{\mathcal{DP}}(x)$). This has been done in \cite{ELY}, where the authors build on and largely extend the results in \cite{PW} by describing a general framework (based on continued fractions) to express the generating functions $G_{T}^{\mathcal{DP}}(x)$. The next proposition is essentially a rephrasing into our language of one of the main results in \cite{ELY}.

\begin{proposition}[\cite{ELY}]\label{eu}
	Let $T$ be a finite subset of $\mathbf{N}$, with $\max T=m$. The generating function of Dyck paths without peaks at height $i$ for all $i\in T$ (with respect to the semilength) can be expressed as a continued fraction as follows:
	\begin{equation}\label{dyck_continued}
		G_{T}^{\mathcal{DP}}(x)=\frac{1}{1+x\delta_T (1)-\frac{x}{1+x\delta_T (2)-\frac{x}{\frac{\ddots}{1+x\delta_T (m-1)-\frac{x}{1+x-xC(x)}}}}},
	\end{equation}
	where (here and in the rest of the section) $C(x)=\frac{1-\sqrt{1-4x}}{2x}$ is the generating function of Catalan numbers and $\delta_T (i)$ is defined by setting $\delta_T (i)=1$ when $i\in T$ and $\delta_T (i)=0$ when $i\notin T$.
\end{proposition}

The above proposition combined with our Theorem \ref{main_th} allows us to compute the $\mex$ generating function of Dyck paths, at least for moderately small values of the $\mex$. Moreover, since the expression in (\ref{dyck_continued}) is rational in $x$ and $C(x)$, Theorem \ref{main_th} immediately implies the following notable corollary.

\begin{corollary}
	The generating function $\Gamma_{m}^{\mathcal{DP}}(x)$ is rational in $x$ and $C(x)$, for all $m\geq 1$.
\end{corollary}

Dyck paths having $\mex$ 1 are Dyck paths with no hills (a hill being a peak having height 1). It is a classical result (see for instance \cite{DS}) that such Dyck paths are enumerated by Fine numbers (sequence A000957 with offset 1 in \cite{OEIS}).

\begin{proposition}
	The generating function of Dyck paths having $\mex$ 1 is
	\[
	\Gamma_{1}^{\mathcal{DP}}(x)=\frac{1}{1+x-xC(x)}.
	\]
\end{proposition}

\begin{proof}
	According to Theorem \ref{main_th}, we simply have $\Gamma_1 (x)=G_{\{ 1\}}^{\mathcal{DP}}(x)$. By Proposition \ref{eu}, we get immediately $G_{\{ 1\}}^{\mathcal{DP}}(x)=\frac{1}{1+x-xC(x)}$, as desired. 
\end{proof}

The case of Dyck paths having $\mex$ 2 is small enough to be handled without the use of a computer. We propose two different derivations of the corresponding generating function, the first one by applying our general methodology, the second one by exploiting a specific combinatorial decomposition of Dyck paths having $\mex$ 2.

\begin{proposition}
	The generating function of Dyck paths having $\mex$ 2 is
	\[
	\Gamma_{2}^{\mathcal{DP}}(x)=\frac{x(1-x)+3x^2 C(x)}{1+x+x^2 C(x)}.
	\]
\end{proposition}

\begin{proof}
	\phantom{a}\\
	
	\emph{First proof.}\quad We apply Theorem \ref{main_th} to get
	\[
	\Gamma_{2}^{\mathcal{DP}}(x)=G_{\{ 2\}}^{\mathcal{DP}}(x)-G_{\{ 1,2\}}^{\mathcal{DP}}(x).
	\]
	
	From Proposition \ref{eu}, also recalling the well-known equation satisfied by the Catalan generating function, i.e. $C(x)=1+xC(x)^2$, we have: 
	\begin{align*}
	G_{\{ 2\}}^{\mathcal{DP}}(x)&=\frac{1}{1-\frac{x}{1+x-xC(x)}}=\frac{1+x-xC(x)}{1-xC(x)}=1+xC(x),\\
	G_{\{ 1,2\}}^{\mathcal{DP}}(x)&=\frac{1}{1+x-\frac{x}{1+x-xC(x)}}=\frac{1+x-xC(x)}{1+x+x^2 -x(1+x)C(x)}=\frac{1+xC(x)}{1+x+x^2 C(x)}.
	\end{align*}
	After a simple computation, the above formulas give $\Gamma_{2}^{\mathcal{DP}}(x)=\frac{x(1-x)+3x^2 C(x)}{1+x+x^2 C(x)}$, as desired. 
	\bigskip
	
	\emph{Second proof.}\quad A Dyck path having $\mex$ 2 (which is necessarily nonempty) can be uniquely decomposed in terms of its first hill: before such hill there is a Dyck path with no peaks at heights 1 and 2, whereas the remaining suffix is a Dyck path without peaks at height 2 (see Figure \ref{dyck_hill}). For the corresponding generating functions we therefore have:
	\[
	\Gamma_{2}^{\mathcal{DP}}(x) = xG_{\{ 1,2\}}^{\mathcal{DP}}(x)G_{\{ 2\}}^{\mathcal{DP}}(x), 
	\]
	which returns the desired expression for $\Gamma_{2}^{\mathcal{DP}}(x)$.
\end{proof}

\begin{figure}[h]
	\centering
	\includegraphics[scale=0.7]{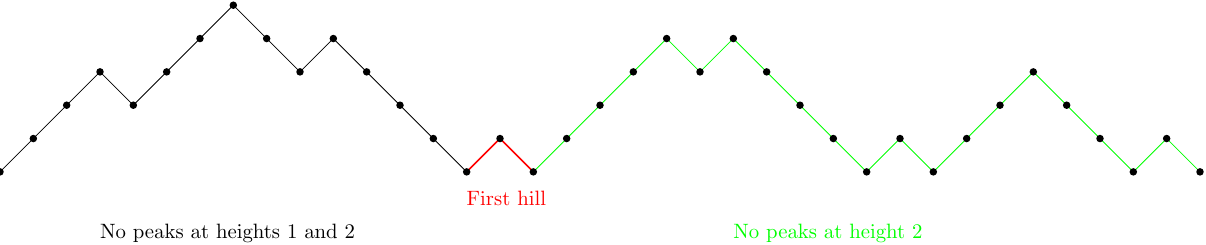}
	\caption{The decomposition of a Dyck path having $\mex$ 2 in terms of its first hill.}
	\label{dyck_hill}
\end{figure}

Using a computer and some symbolic computation software, it is possible to deduce the generating function $\Gamma_{m}^{\mathcal{DP}}(x)$ also for moderately larger values of $m$. For instance, for $m=3$, we get
{\footnotesize
\[
\Gamma_{3}^{\mathcal{DP}}(x)=\frac{x^3 (2-x-5x^3 -2x^4-x^6)-x^4 (2+x-2x^3-2x^4-3x^5)C(x)}{(1-3x-x^2+3x^3+2x^4+2x^5+x^6-x^7)-x(1-2x-2x^2+x^3+x^4+2x^5+x^6+x^7)C(x)}.
\] 
}

\section{Set partitions}\label{SP}

Setting $w(B)=|B|$ for all $B\in \powerset_{fin}(\mathbf{N})$, the mex of a set partition $\pi =(B_1 ,B_2 ,\ldots ,B_k )$ of size $n$ is the minimum positive integer $k$ such that $|B_i |\neq k$, for all $i$. In order to provide an effective description of the mex statistic, we make use of \emph{exponential} generating functions (rather than ordinary ones, like in all previous cases). 

It is very easy to see that our approach to the mex statistic can also be pursued for exponential generating functions, and in particular the statement of Theorem \ref{main_th} remains unchanged. Therefore in the rest of the present section we will denote with $G_{T}^{\mathcal{SP}}(x)$ the exponential generating function of set partitions having no block of cardinality $i$, for all $i\in T$ (where $T$ is a finite subset of positive natural numbers) and with $\Gamma_m ^{\mathcal{SP}}(x)=\sum_{n\geq 0}\gamma^{\mathcal{SP}}_{n,m}\frac{x^n}{n!}$ the exponential generating function of set partitions having mex $m$ (where $m$ is a positive natural number). We provide a general expression for the exponential generating function $G_{T}^{\mathcal{SP}}(x)$. We remark that our result is related to \cite{BeRa}, where set partitions with restricted block sizes are considered (even if our approach is slightly different).

\begin{proposition}
	 The exponential generating function $G_{T}^{\mathcal{SP}}(x)$ is
	 \[
	  G_{T}^{\mathcal{SP}}(x)=e^{e^x -1-\sum_{n\in T}\frac{x^n}{n!}}.
	 \]
\end{proposition}

\begin{proof}
	Our starting point is a (formal differential) equation satisfied by $G_{T}^{\mathcal{SP}}(x)$ which adapts the classical argument used to determine the exponential generating function of unrestricted set partitions to our specific case. In order to uniquely determine a nonempty set partition whose blocks avoid the cardinality listed in $T$, we may fix an element of the support and consider all possible blocks of cardinality $k+1$ containing that element (where of course $k+1\notin T$), then we need to take a set partition on the remaining $n-k$ elements (again having blocks avoiding the cardinalities in $T$). This construction translates into the differential equation
	\[
	 (G_{T}^{\mathcal{SP}})'(x)=H_T (x)\cdot G_{T}^{\mathcal{SP}}(x),
	\]
	where $H_{T}(x)=e^x -\sum_{n\in T}\frac{x^{n-1}}{(n-1)!}$. In order to solve this differential equation, we write it as a logarithmic derivative:
	\[
	 (\log (G_{T}^{\mathcal{SP}}(x)))'=\frac{(G_{T}^{\mathcal{SP}}(x))'}{G_{T}^{\mathcal{SP}}(x)}=H_T (x),
	\]
	hence
	\[
	 \log G_{T}^{\mathcal{SP}}(x)=e^x -\sum_{n\in T}\frac{x^n}{n!}+c.
	\]
	Since the empty set partition is enumerated by the generating function $G_{T}^{\mathcal{SP}}(x)$, we get that $c=-1$, and so
	\[
	 G_{T}^{\mathcal{SP}}(x)=e^{e^x -1-\sum_{n\in T}\frac{x^n}{n!}},
	\]
	as desired.
\end{proof}

\begin{corollary}
	The exponential generating function $\Gamma_{m}^{\mathcal{SP}}(x)$ of set partitions having $\mex$ $m$ can be explicitly expressed as
	\[
	 \Gamma_{m}^{\mathcal{SP}}(x)=e^{e^x -1-\frac{x^m}{m!}} \prod_{k=1}^{m-1}(1-e^{-\frac{x^k}{k!}}).
	\]
\end{corollary}

\begin{proof}
	Theorem \ref{main_th} allows us to explicitly determine $\Gamma_{m}^{\mathcal{SP}}(x)$. Indeed, using also the expression for $G_{T}^{\mathcal{SP}}(x)$ found in the previous proposition, we get
	\begin{align*}
	 \Gamma_{m}^{\mathcal{SP}}(x)&=\sum_{k=1}^{m}(-1)^{k+1}\sum_{{T\subseteq [m]\atop |T|=k}\atop m\in T}G_{T}^{\mathcal{SP}}(x) \\
	 &=\sum_{k=1}^{m}(-1)^{k+1}\sum_{{T\subseteq [m]\atop |T|=k}\atop m\in T}e^{e^x -1-\sum_{n\in T}\frac{x^n}{n!}} \\
	 &=e^{e^x -1-\frac{x^m}{m!}}\sum_{k=1}^{m}(-1)^{k-1}\sum_{T\subseteq [m-1]\atop |T|=k-1}e^{-\sum_{n\in T}\frac{x^n}{n!}} \\
	 &= e^{e^x -1-\frac{x^m}{m!}}\prod_{k=1}^{m-1}(1-e^{-\frac{x^k}{k!}}),
	\end{align*}
	as desired.
\end{proof}

For small values of $m$ we can easily get the corresponding sequences of coefficients.

For $m=1$, we obtain $\Gamma_{1}^{\mathcal{SP}}(x)=e^{e^x -1-x}$, that is the generating function of set partitions without singleton blocks (sequence A000296 in \cite{OEIS}, whose sequence of coefficients starts $1,0,1,1,4,11,41,162,715,3425,17722,98253,\ldots$).

For $m=2$, we obtain $\Gamma_{2}^{\mathcal{SP}}(x)=e^{e^x -1-\frac{x^2}{2}}(1-e^{-x})$, whose sequence of coefficients starts $0,1,1,1,5,16,42,169,779,$ $3385,16263\ldots$ and is not recorded in \cite{OEIS}.

For $m=3$, we obtain $\Gamma_{3}^{\mathcal{SP}}(x)=e^{e^x -1-\frac{x^3}{6}}(1-e^{-x})(1-e^{-\frac{x^2}{2}})$, whose sequence of coefficients starts $0,0,0,3,6,25,60,336,1246,6777,29070,\ldots$ and is not recorded in \cite{OEIS}.

%\begin{corollary}
%	\begin{enumerate}
%		\item $\Gamma_{1}^{\mathcal{SP}}(x)=e^{e^x -1-x}$, that is the generating function of set partitions without singleton blocks (sequence A000296 in \cite{OEIS}, whose sequence of coefficients starts $1,0,1,1,4,11,41,162,715,3425,17722,98253,\ldots$).
%		\item $\Gamma_{2}^{\mathcal{SP}}(x)=e^{e^x -1-\frac{x^2}{2}}(1-e^{-x})$, whose sequence of coefficients starts $0,1,1,1,5,16,42,169,779,$ $3385,16263\ldots$ and is not recorded in \cite{OEIS}.
%	\end{enumerate}
%\end{corollary}

\section{Planar trees}\label{PT}
	
According to our presentation of planar trees as given in Section \ref{main_def_and_prop}, and defining the weight function $w$ by setting $w(A)=|A|-1$ for all $A\in \powerset_{fin}(\mathbf{N})$, the $\mex$ of a planar tree $\tau$ is the minimum integer $m$ such that $\tau$ does not have vertices having $m$ children. In order to find the generating function $G_{T}^{\mathcal{PT}}(x)$ of planar trees having no vertices with $k$ children, for all $k\in T$, we exploit the usual decomposition of planar trees in terms of the subtrees whose roots are the children of the root.

\begin{proposition}\label{G_PT}
	Given $T\subseteq \mathbf{N}$, we get
	\[
	G_{T}^{\mathcal{PT}}(x)=1+x+x\cdot \sum_{k\notin T}(G_{T}^{\mathcal{PT}}(x)-1)^k .
	\]
\end{proposition}

\begin{proof}
	A nonempty planar tree such that each node does not have $k$ children, for all $k\in T$, either has a single node or it can be decomposed as the root together with the subtrees generated by each of its children. All these (nonempty) subtrees have the property that their nodes do not have $k$ children, for all $k\in T$, and the number of subtrees of the root is clearly an integer not belonging to $T$. All the above conditions translate into the desired functional equation satisfied by $G_{T}^{\mathcal{PT}}(x)$.    
\end{proof}   

\begin{corollary}
	The generating function $\Gamma_{m}^{\mathcal{PT}}(x)$ of planar trees having $\mex$ $m$ is algebraic.
\end{corollary}

The case of $\mex$ 1 can be dealt with rather easily. Indeed, planar trees having $n$ nodes and no nodes with only 1 child are known to be counted by Riordan numbers (sequence A005043 in \cite{OEIS}). We can reprove this fact using our methodology, which allows us to easily find the generating function.

\begin{proposition}
	The generating function of planar trees having $\mex$ 1 is
	\[
		\Gamma_{1}^{\mathcal{PT}}(x)=\frac{3(1+x)-\sqrt{1-2x-3x^2}}{2(1+x)}.	
	\]	
\end{proposition}

\begin{proof}
	When $T=\{ 1\}$ we have as usual $\Gamma_{1}^{\mathcal{PT}}(x)=G_{\{ 1\}}^{\mathcal{PT}}(x)$. From Proposition \ref{G_PT} we get
	\begin{align*}
	G_{\{ 1\}}^{\mathcal{PT}}(x)&=1+x+x\sum_{k\geq 2}(G_{\{ 1\}}^{\mathcal{PT}}(x)-1)^k \\
	&=1+x+x\frac{(G_{\{ 1\}}^{\mathcal{PT}}(x)-1)^2}{2-G_{\{ 1\}}^{\mathcal{PT}}(x)},
	\end{align*}
	which leads to the algebraic equation
	\[
	(1+x)G_{\{ 1\}}^{\mathcal{PT}}(x)^2 -3(1+x)G_{\{ 1\}}^{\mathcal{PT}}(x)+2+3x=0
	\]
	whose solution is
	\[
	G_{\{ 1\}}^{\mathcal{PT}}(x)=\frac{3(1+x)-\sqrt{1-2x-3x^2}}{2(1+x)},
	\]
	as desired.
\end{proof}

The case of $\mex$ 2 is more complicated and requires the use of Lagrange inversion formula. Since there are many equivalent formulations of Lagrange inversion formula, we record here the one we will use in the subsequent proofs.

\begin{theorem}[Lagrange inversion formula]\label{theor:lagr}
	Let $F(x), H(x)$ be formal power series, with $F(0)\neq 0$, such that the following equality holds:
	\[
	H(x)=xF(H(x)).
	\]
	Then, for all $m,k\geq 0$, with $m\geq k$, we have:
	\begin{equation}\label{lagrange}
	m[x^m ]H(x)^k =k[x^{m-k}]F(x)^m ,
	\end{equation}
	where, as usual, the notation $[x^n]R(x)$ denotes the $n$-th coefficient of the formal power series $R(x)$.
\end{theorem}

\begin{lemma}
	For $n\geq 1$, the $n$-th coefficient of the generating function $G_{\{ 2\}}^{\mathcal{PT}}(x)=\sum_{n\geq 0}g_{\{ 2\}} ^{\mathcal{PT}}(n)x^n$ of planar trees having no nodes with exactly 2 children is
		\[
		g_{\{ 2\}} ^{\mathcal{PT}}(n)=\frac{1}{n}\sum_{k=0}^{n}(-1)^k \binom{n}{k}\binom{2n-3k-2}{n-k-1}.
		\]
\end{lemma}

\begin{proof}
	From Proposition \ref{G_PT}, when $T=\{ 2\}$ we have the following equation satisfied by $G_{\{ 2\}}^{\mathcal{PT}}(x)$:
	\[
	G_{\{ 2\}}^{\mathcal{PT}}(x)-1=x\cdot \left( 1+(G_{\{ 2\}}^{\mathcal{PT}}(x)-1)+(G_{\{ 2\}}^{\mathcal{PT}}(x)-1)^3 \frac{1}{1-(G_{\{ 2\}}^{\mathcal{PT}}(x)-1)} \right) .
	\]
	
	We can use Lagrange inversion formula by choosing $H(x)=G_{\{ 2\}}^{\mathcal{PT}}(x)-1$ and $F(x)=1+x+\frac{x^3}{1-x}$. If $k=1$ and $m\geq 1$, expression (\ref{lagrange}) then becomes
	\begin{equation}\label{lagrange_2}
	g_{\{ 2\}} ^{\mathcal{PT}}(m)=[x^m ]G_{\{ 2\}}^{\mathcal{PT}}(x)=\frac{1}{m}[x^{m-1}]\left( 1+x+\frac{x^3}{1-x} \right)^m .
	\end{equation}
	
	In order to determine the coefficients of the formal power series $\sum_{n\geq 0}f(n)x^n =\left( 1+x+\frac{x^3}{1-x}\right)^m$, we observe that $f(n)$ can be interpreted as the number of weak integer compositions of $n$ into $m$ parts with no parts equal to 2. As a consequence, we can use an inclusion/exclusion argument to obtain a closed form for $f(n)$. Indeed, set $A_i =\{ (\lambda_1 ,\lambda_2 ,\ldots ,\lambda_m )\, |\, \lambda_1 +\lambda_2 +\cdots +\lambda_m =n, \lambda_i =2\}$, for $i=1,2,\ldots ,m$ (i.e., $A_i$ is the set of weak compositions of $n$ into $m$ parts in which the $i$-th part is 2). Then
	\[
	f(n)=|A'_1 \cap A'_2 \cap \cdots \cap A'_m |=\sum_{k=0}^{m}(-1)^k \sum_{T\subseteq [m]\atop |T|=k}\left| \bigcap_{i\in T}A_i\right| .
	\]

	We can compute the quantities $\left| \bigcap_{i\in T}A_i\right|$ by observing that they count weak compositions of $n$ into $m$ parts in which the $i$-th part is 2, for all $i\in T$ (where $|T|=k$). It is immediate to see that the above set of weak compositions is in bijection with the set of weak compositions of $n-2k$ into $m-k$ parts (just remove the parts whose indices belong to $T$), and these are easily seen to be $\binom{n+m-3k-1}{m-k-1}$. Since these quantities only depend on $|T|=k$ (and not on $T$), we finally get the following expression for $f(n)$:
	\[
	f(n)=\sum_{0}^{m}(-1)^k \binom{m}{k}\binom{n+m-3k-1}{m-k-1}.
	\]
	
	Now plugging $f(m-1)$ into (\ref{lagrange_2}) we finally get an expression for $g_{\{ 2\}} ^{\mathcal{PT}}(m)$:
	\[
	g_{\{ 2\}} ^{\mathcal{PT}}(m)=\frac{1}{m}f(m-1)=\frac{1}{m}\sum_{k=0}^{m}(-1)^k \binom{m}{k}\binom{2m-3k-2}{m-k-1},
	\]
	which is the desired one.	 
\end{proof}

\begin{lemma}
	For $n\geq 1$, the $n$-th coefficient of the generating function $G_{\{ 1,2\} }^{\mathcal{PT}}(x)=\sum_{n\geq 0}g_{\{ 1,2\}} ^{\mathcal{PT}}(n)x^n$ of planar trees having no nodes with 1 or 2 children is
	\[
	g_{\{ 1,2\} }^{\mathcal{PT}}(n)=\frac{1}{n}\sum_{k=0}^{n}\binom{n}{k}\binom{n-2k-2}{k-1}.
	\]
\end{lemma}

\begin{proof}
	Using the same approach as in the proof of the previous lemma, we can exploit Proposition \ref{G_PT} to find a functional equation satisfied by $G_{\{ 1,2\} }^{\mathcal{PT}}(x)$, namely (after simplifying and regrouping)
	\[
	G_{\{ 1,2\} }^{\mathcal{PT}}(x)-1=x\cdot \left( 1+\frac{(G_{\{ 1,2\} }^{\mathcal{PT}}(x)-1)^3}{1-(G_{\{ 1,2\} }^{\mathcal{PT}}(x)-1)} \right) .
	\]
	
	Invoking again Lagrange inversion formula (\ref{lagrange}) with $H(x)=G_{\{ 1,2\} }^{\mathcal{PT}}(x)-1$ and $F(x)=1+\frac{x^3}{1-x}$ we then get (for $m\geq 1$):
	\begin{equation}\label{lagrange_12}
	g_{\{ 1,2\} }^{\mathcal{PT}}(m)=[x^m ]G_{\{ 1,2\} }^{\mathcal{PT}}(x)=\frac{1}{m}[x^{m-1}]\left( 1+\frac{x^3}{1-x} \right)^m .
	\end{equation}
	
	We observe that, in this case, the $n$-th coefficient $f(n)$ of the series $\sum_{n\geq 0}f(n)x^n =\left( 1+\frac{x^3}{1-x} \right)^m$ counts weak compositions of $n$ into $m$ parts with no part equal to 1 or 2. Instead of using inclusion/exclusion (which is not convenient in this case), we find a closed form for $f(n)$ by working directly on the generating function:
	\[
	\left( 1+\frac{x^3}{1-x} \right)^m =\sum_{k=0}^{m}\binom{m}{k}x^{3k}\frac{1}{(1-x)^k}=\sum_{k=0}^{m}\binom{m}{k}x^{3k}\sum_{h\geq 0}\binom{h+k-1}{k-1}x^h ,
	\]
	hence
	\[
	f(n)=[x^n ]\left( 1+\frac{x^3}{1-x} \right)^m =\sum_{k=0}^{n}\binom{m}{k}\binom{n-2k-1}{k-1}.
	\]
	
	Plugging the above expression into (\ref{lagrange_12}) gives the desired result for $g_{\{ 1,2\} }^{\mathcal{PT}}(m)$.
\end{proof}

Putting together the results found in the last two lemmas gives an expression for the number of planar trees with $n$ nodes having $\mex$ 2.

\begin{corollary}
	The $n$-th coefficient of the generating function $\Gamma_{2}^{\mathcal{PT}}(x)=G_{\{ 2\} }^{\mathcal{PT}}(x)-G_{\{ 1,2\} }^{\mathcal{PT}}(x)$ is
	\[
	\frac{1}{n}\left( \sum_{k=0}^{n}(-1)^k \binom{n}{k}\binom{2n-3k-2}{n-k-1}-\sum_{k=0}^{n}\binom{n}{k}\binom{n-2k-2}{k-1}\right) ,
	\]
	thus obtaining the sequence $1,1,1,1,5,16,42,120,359, 1078, 3278, 10077, 31209, 97449, 306568,\ldots$ which is not recorded in \cite{OEIS}.
\end{corollary}

\section{Further work}\label{further}

\begin{enumerate}
	
	\item Concerning inversion sequences, looking at the first values of $\gamma_{n,m}^{\mathcal{IS}}$ as shown in the table at the end of Section \ref{IS}, it turns out that the resulting terms appears to match the table recorded as A056151 in \cite{OEIS} (with reversed rows). Such a sequence describes the distribution of the maximum in inversion sequences. We have not been able to prove this striking coincidence.
	
	\item The $\mex$ statistic can be investigated also on many other combinatorial objects, such as ordered set partitions, permutations, polyominoes, finite posets and finite graphs, to mention just a few. For each type of object, a suitable weighted combinatorial structure should be defined, whose pieces determine the associated notion of $\mex$. We would like to remark that, even in the cases we have considered in the present paper, if we endow a set of objects with a different weighted combinatorial structure, in general a different notion of $\mex$ may arise, which could be interesting to explore. As an example, we may define another combinatorial structure on Dyck paths, by defining a Dyck path as a sequence $((r_1 ,h_1),(r_2 ,h_2 ),\ldots ,(r_n ,h_n))$ of pairs of nonnegative integers, where $r_i$ is interpreted as the length of the $i$-th run of up steps of the path, and $h_i$ is the height at which such a run starts . Without delving into details, this can be done by taking the pairs of nonnegative integers as the set of pieces, and then imposing a suitable set of properties. The $\mex$ statistic determined by the weighted combinatorial structure in which the weight of a piece is its first component associates with every Dyck path the minimum integer $k$ such that the path does not contain maximal runs of $k$ consecutive up steps.
	
	\item The $\mex$ function defined in \cite{AN} for integer partitions is actually a more general version of our $\mex$. Indeed, for an integer partition $\lambda$, Andrews and Newman define $\mex_{A,a} (\lambda )$ to be the smallest integer congruent to $a$ modulo $A$ that is not a part of $\lambda$. Therefore our $\mex$ corresponds to $\mex_{1,1}$. In \cite{AN} it is shown that several interesting partition identities are related with the $\mex$ function. In particular, the crank statistic for integer partition has some relationship with $\mex_{1,1}$. It would be interesting to explore whether analogous number-theoretic properties hold also for the $\mex$ function of other combinatorial structure, and in particular if the notion of crank for integer partitions has some analogous in other combinatorial contexts.
	
	\item The importance of the classical $\mex$ function in combinatorial game theory is well-known, and it has been briefly discussed in the Introduction. Are there any examples in which the $\mex$ function of some combinatorial structure (as we have defined it) might be useful in the context of combinatorial games?      

\end{enumerate}

\end{document}